\documentclass[onefignum,onetabnum]{siamonline220329}

\usepackage{amsfonts}
\usepackage{amssymb}
\usepackage{amsmath}
\usepackage{hyperref}
\usepackage{booktabs} % For pretty tables
\usepackage{caption} % For caption spacing
\usepackage{subcaption} % For sub-figures
\usepackage{graphicx}
\usepackage{pgfplots}
\usepackage[all]{nowidow}
\usepackage[utf8]{inputenc}
\usepackage{tikz}
\usetikzlibrary{er,positioning,bayesnet}
\usepackage{multicol}
\usepackage{algpseudocode,algorithm,algorithmicx}

\usepackage{epstopdf}
\ifpdf
  \DeclareGraphicsExtensions{.eps,.pdf,.png,.jpg}
\else
  \DeclareGraphicsExtensions{.eps}
\fi

\newsiamremark{assumption}{Assumption}
\newsiamremark{remark}{Remark}

\makeatletter
\AtBeginDocument{
  \def\refstepcounter@optarg[#1]#2{%
    \cref@old@refstepcounter{#2}%
    \cref@constructprefix{#2}{\cref@result}%
    \@ifundefined{cref@#1@alias}%
    {\def\@tempa{#1}}%
    {\def\@tempa{\csname cref@#1@alias\endcsname}}%
    \protected@edef\cref@currentlabel{%
      [\@tempa][\arabic{#2}][\cref@result]%
      \csname p@#2\endcsname\csname the#2\endcsname}}
}
\makeatother

\newcommand{\fa}{\mathbb{\forall}}

\newcommand{\R}{\mathbb{R}}
\newcommand{\N}{\mathbb{N}}

\newcommand{\Ra}{\Rightarrow} % =>
\newcommand{\subspace}{\hookrightarrow}

\newcommand{\C}{\mathbb{C}}

\newcommand{\cl}{\overline}

\newcommand{\weakto}{\rightharpoonup}
\newcommand{\bd}{\partial}
\headers{Variational Lippmann Schwinger equation}{A. Tataris and A.V. Mamonov}

\title{A variational Lippmann-Schwinger-type approach for the Helmholtz impedance problem on bounded domains
\thanks{Submitted to SIMA.
\funding{A. V. Mamonov was supported by the U.S. National Science
Foundation under award DMS-2309197. This material is based upon research supported in
part by the U.S. Office of Naval Research under award number N00014-21-1-2370 to A. V. Mamonov.}}}

\author{Andreas Tataris\thanks{Delft University of Technology, The Netherlands 
  (\email{A.Tataris@tudelft.nl}).} 
\and Alexander V. Mamonov\thanks{Department of Mathematics, University of Houston, Houston, TX, USA 
  (\email{avmamonov@uh.edu}).}}

\usepackage{amsopn}

\ifpdf
\hypersetup{
  pdftitle={Inverse scattering for Schr\"{o}dinger equation via ROM}{A. Tataris, T. van Leeuwen, and A.V. Mamonov},
  pdfauthor={A. Tataris, T. van Leeuwen, and A.V. Mamonov}
}
\fi

\begin{document}

\maketitle

% REQUIRED
\begin{abstract}
Recently, reduced order modeling methods have been applied to solving inverse boundary value problems arising in frequency domain scattering theory.
    A key step in projection-based reduced order model methods is the use of a sesquilinear form associated with the forward boundary value problem. 
    However, in contrast to scattering problems posed in $\R^d$, boundary value formulations lose certain structural properties, most notably the classical Lippmann-Schwinger integral equation is no longer available. 
    In this paper we derive a Lippmann-Schwinger type equation aimed at studying the solution of a Helmholtz  boundary value problem with a variable refractive index and impedance boundary conditions. In particular, we start from the variational formulation of the boundary value problem and we obtain an equivalent operator equation which can be viewed as a bounded domain analogue of the classical Lippmann-Schwinger equation. We first establish analytical properties of our variational Lippmann-Schwinger type operator. Based on these results, we then show that the parameter-to-state map, which maps a refractive index to the corresponding wavefield, maps weakly convergent sequences to strongly convergent ones when restricted to refractive indices in Lebesgue spaces with exponent greater than 2.
    %make it as motiv.
    Finally, we use the derived weak to strong sequential continuity to show existence of minimizers for a reduced order model based optimization methods aimed at solving the inverse boundary value problem as well as for a conventional data misfit based waveform inversion method.
\end{abstract}

% REQUIRED
\begin{keywords}
%Inverse scattering, full waveform inversion, 
Helmholtz equation, Lippmann-Schwinger equation, 
frequency domain,
model order reduction.
\end{keywords}

% REQUIRED
\begin{MSCcodes}
46N20, 35J05, 49J20
%65M32, 41A20
\end{MSCcodes}

%===========================================================

%===========================================================
\section{Introduction}
% mention isp vs ibvp and 
Time-harmonic wave scattering, arises in a variety of physical settings, including acoustic, electromagnetic and elastic wave propagation, see \cite{Colton1992InverseAA,devaney2012mathematical, tataris2022regularised, GLM_Helmholtz_Andreas}.
Scattering problems posed on the entire space $\R^d, \ d=2,3$ are typically treated via the so called Lippmann-Schwinger integral equation \cite{Colton1992InverseAA}. The Lippmann-Schwinger equation provides a structured framework and therefore great flexibility in studying the solutions of the forward scattering problem as well as studying the inverse scattering problem. In contrast, when we pose the same scattering problem on a bounded domain then the classical Lippmann-Schwinger equation is no longer available. This creates challenges in studying forward and inverse boundary value problems unlike the $\R^d$ scattering case.% Given the recent advances in projection based reduced order models for solving inverse boundary value problems, having an analogue of the Lippmann-Schwinger equation can be a powerful tool since projection based roms are inherently based on boundary value formulations of PDEs. 

In this paper, we consider the scattering of time-harmonic waves, modeled by the Helmholtz equation with impedance boundary conditions on a bounded domain $\Omega \subset \R ^2$. The problem is formulated as
\begin{align}
    (-\Delta -k^2 m)u=f,\text{ in } \Omega
    \label{eq: Helm},
\end{align}
\begin{align}
    (\partial_n -\imath k )u=0 
    , \text{ on } \bd \Omega,
    \label{eq: BC}
\end{align}
for a given wavenumber $k>0$, volume source $f$ and refractive index $m$.
This boundary value formulation can be viewed as the bounded-domain analogue of the classical Helmholtz scattering problem with Sommerfeld radiation condition in $\R^2$. %, which is treated using the Lippmann–Schwinger integral equation. 
In the $\R^2$ setting, one seeks for a total field $u$ that satisfies the (classical) Lippmann–Schwinger integral equation
\begin{align}
    u(x)=u^{inc}(x)+k^2\int_{\R^2} \Phi_k(x,y) u(y)m(y)dy, \ x\in \R^{2} ,
    \label{eq: classical LS of intro}
\end{align}
where $u^{inc}$ is an incident field and $\Phi_k$ is the fundamental solution of the unperturbed Helmholtz operator, $-\Delta-k^2$.
The forward scattering problem is therefore equivalent to solving a Fredholm equation of the second kind for the total field $u(x), \ x\in \R^d$. 
Working within this classical Lippmann–Schwinger framework offers several advantages when it comes to studying analytical properties of the solutions of the forward scattering problem. In particular, in \cite{lechleiter2013tikhonov} has been shown that the volume integral formulation of the scattering problem yields weak sequential compactness and closedness of the parameter-to-wavefield map. These properties allow to show well posedness results of full waveform inversion type reconstruction methods aimed to solve the inverse scattering problem.

Recently, data-driven reduced order models (ROM) have emerged as powerful tools used for waveform inversion on bounded domains, see \cite{borcea2020reduced,borcea2021reduced,borcea2023waveform1,borcea2023waveform2,borcea2024data,borcea2025reduced,druskin2016direct,druskin2018nonlinear,tataris2023reduced,tataris2025inversescatteringschrodingerequation}.
These methods typically involve assembling misfit functionals using volume inner products of wavefields and require the variational (weak) formulation of the forward problem. 
However, the classical Lippmann–Schwinger equation is not applicable in the bounded domain setting. This motivates the need for an analogue of the Lippmann-Schwinger equation tailored for bounded domains. 
Our aim is to construct such a framework, enabling analysis of the parameter-to-wavefield map in the spirit of \cite{lechleiter2013tikhonov}, and providing a theoretical foundation for ROM-based inversion methods.
In an attempt to obtain an integral equation setup for the bounded domain setting described by the boundary value problem \eqref{eq: Helm}-\eqref{eq: BC}, one has to make use of the integral representation of smooth functions, see \cite{Colton1992InverseAA}. 
Assuming a sufficiently smooth solution, $u,$ of the boundary value problem 
\eqref{eq: Helm}-\eqref{eq: BC}, we obtain the following integral representation 
\begin{align}
    u(x)=\int_{\bd \Omega}  \big(\imath k   \Phi_k(x,y)- \bd_n \Phi_k(x,y)\big) u(y )ds(y)+ \nonumber \\ \int_{\Omega} \big(k^2q(y) u(y)+f(y)\big)  \Phi_k(x,y) dy, \ x
\in 
\Omega,
\label{eq: IE1}
\end{align}
with $m=1+q.$
This integral equation can be extended to the boundary $\bd \Omega$ using the extension properties of the boundary integral operators (single and double layer potentials). This extension leads to the following integral equation
\begin{align}
    u(x)=\int_{\bd \Omega}  \Big(\imath k   \Phi_k(x,y)- \bd_n \Phi_k(x,y)+\frac{1}{2}\Big) u(y )ds(y)+ \nonumber \\ \int_{\Omega} \Big(k^2q(y) u(y)+f(y)\Big)  \Phi_k(x,y) dy, \ x
\in 
\bd\Omega.
\label{eq: IE2}
\end{align}
The coupled system of integral equations \eqref{eq: IE1} and \eqref{eq: IE2} is equivalent to the original boundary value problem \eqref{eq: Helm}-\eqref{eq: BC}. 
One possible approach for studying the parameter-to-wavefield map is to work with this coupled system, and follow an analysis in the spirit of \cite{lechleiter2013tikhonov}.
However, showing analytical properties and well‐posedness for the coupled system of integral equations \eqref{eq: IE1}-\eqref{eq: IE2}, is far from straightforward. In particular, one must show that this coupled system  forms a bounded, invertible operator on appropriate function spaces. To the best of our knowledge, showing well posedness for this system remains an open problem.
%Another potential option is to derive analytical properties of the parameter-to-solution map by studying the inverse of the operator that describes the variational formulation of \eqref{eq: Helm}–\eqref{eq: BC} as a function of the refractice index. However, this approach is impractical since it involves a highly nonlinear operator.

In this paper we propose an alternative way to study the functional analytic properties of the nonlinear parameter-to-solution map of the Helmholtz boundary value problem \eqref{eq: Helm}-\eqref{eq: BC}. In particular, we derive an analogue of the Lippmann-Schwinger equation  but in a variational sense. We begin from the weak formulation of the Helmholtz problem and we derive an equivalent Fredholm linear operator equation. 
We achieve that by applying the inverse of the operator that describes the weak formulation of \eqref{eq: Helm}-\eqref{eq: BC} when $m=1$, to the operator equation that corresponds to the weak formulation of the Helmholtz boundary value problem with variable $m$. Following this, we study the analytical properties of the resulting linear operators that form a Lippmann-Schwinger type equation. Using the derived results we establish analytic properties of the parameter-to-state map such as its weak-to-strong sequential continuity. Finally, we investigate well posedness results of optimization problems associated to conventional full waveform inversion or reduced order model (ROM) inversion for solving the inverse boundary value problem.
%Our approach offers an alternative framework for studying the solutions of the boundary value problem \eqref{eq: Helm}–\eqref{eq: BC} and their dependence on the refractive index, via the introduced variational Lippmann–Schwinger equation. % This way we avoid the previously mentioned technically involved approaches, such as the analysis of the coupled system \eqref{eq: IE1}–\eqref{eq: IE2} or the direct study of the inverse of the variational Helmholtz operator as a nonlinear function of the refractive index.

This paper is organized as follows. We begin with Section \ref{sec: Prelims} where we formulate the Helmholtz impedance boundary value problem and we state our assumptions on the domain of interest and on the refractive index. In Section \ref{sec: main results} we present our main contributions. We begin with deriving the variational Lippmann-Schwinger equation and we show analytical properties of associated operators. Based on these we show sequential weak-to-strong continuity of the parameter-to-wavefield map.
We then use the derived continuity property of the solution map to show existence of minimizers for  reduced order model based and conventional full waveform inversion methods for solving the inverse boundary value problem. We conclude the paper with an outlook and conclusion section.

\section{Preliminaries \label{sec: Prelims}}
In this section, we review some basic properties of the Helmholtz impedance boundary value problem. In particular, we give its variational formulation and we review some regularity properties of the solution.  

For reader's convenience we explain the notation we use. With $H^s(\Omega)$ we denote the usual Sobolev space with index $s$, and with $L^p(\Omega)$ the Lebesgue space of exponent $p.$ With $\overline{H^1(\Omega)}'$ we denote the anti-dual of the Sobolev space $H^1(\Omega)$, see \cite{Treves1967}. We denote the action of an element $f\in \overline{H^1(\Omega)}'$ on $\phi \in H^1(\Omega)$ as $\langle f, \phi\rangle \in \C.$ The embedding operator between two function spaces will be denoted as $id.$ 

Before starting we make the following assumptions.
\begin{assumption}
    \begin{enumerate}
        \item The set $\Omega$ is bounded and sufficiently  smooth, at least $C^1$, such that the following embedding between Sobolev spaces is compact 
        \begin{align}
            id:H^2(\Omega) \subspace H^1(\Omega).
        \end{align}
        \item We take $m=1+q$, with $q \in L^p(\Omega;[-1,\infty))$ for some (fixed) $ p \in(2,\infty].$ 
        % not needed \item The source term, $f$, is assumed to be of $L^2-$regularity.
    \end{enumerate}
    \label{assumption}
\end{assumption}

\subsection{Helmholtz impedance boundary value problem and $H^2-$regularity estimates \label{sub: BVP and regularity}}

Keeping in mind the requirements of assumption \ref{assumption}, given a wavenumber $k>0$ and a distribution $f\in \overline{H^1(\Omega)}'$, the boundary value problem \eqref{eq: Helm}-\eqref{eq: BC}
admits a variational (weak) formulation of 
finding $u\in H^1(\Omega)$ such that
\begin{align}
    \fa v\in H^1(\Omega) : \ \int_\Omega \nabla  u \cdot \cl{\nabla v}dx -k^2  \int_\Omega   m u\cl{ v}dx -\imath k \int_{\bd \Omega}    u  \cl{ v}d\Sigma(x) = \langle f,v\rangle .
    \label{eq: weak form}
\end{align}
The existence and uniqueness of  the solution of
\eqref{eq: weak form} is guaranteed by the following result, see \cite{Wald2018}.
\begin{proposition}
Given a wavenumber $k>0, \ f\in\overline{H^1(\Omega)}' $ and $m$ according to assumption \ref{assumption}, the problem \eqref{eq: weak form} admits a unique solution.
\label{pr: forward problem well posedness}
\end{proposition}
\begin{proof}
Here we give a brief sketch of the proof.
For a more detailed sketch of the proof  we refer to the Appendix \ref{App: Well posedness}, and full details can be found in \cite{Bao_2005,Wald2018}. 
Using the weak formulation of the PDE we define the forms
\begin{align}
a_1(u,v)=\int_\Omega \nabla u  \cdot \overline{\nabla v}dx-\imath k  \int_{\bd \Omega} u\overline{v}d\Sigma (x), \ u,v\in H^1(\Omega),
\end{align}
and
\begin{align}
    %\langle \mathcal V u,v\rangle =
    a_2(u,v)=-\int_\Omega mu \overline{v}dx,\ u,v\in H^1(\Omega).
\end{align}
Through the forms $a_1,a_2$ we define operators 
$$\mathcal T:H^1(\Omega)\to H^1(\Omega) \text{ as } (\mathcal Tu,v)_{H^1}=a_1(u,v), \ u,v\in H^1(\Omega)$$
and $$\mathcal W: L^2(\Omega) \to \overline{H^1(\Omega)}', \ u\stackrel{\mathcal W }{\mapsto} a_2(u,\cdot).$$
In short, the Helmholtz boundary value problem can be reduced to the following linear problem of finding $u\in H^1(\Omega)$ that satisfies 
\begin{align}
    \Phi \mathcal {T} ({I}+k^2\mathcal{A})u=f ,   \text{ in } \overline{H^1(\Omega)}',
\end{align}
with $\Phi:H^1(\Omega)\to \overline{H^1(\Omega)}',$ being the linear Riesz isomorphism
and $$\mathcal A:=\mathcal T^{-1}\Phi^{-1} \mathcal W \  id_{H^1\to L^2}:H^1(\Omega) \to H^1(\Omega),$$
where $id_{H^1\to L^2}$ is the compact imbedding operator of $H^1(\Omega)$ into $L^2(\Omega).$
\end{proof}
We quickly note that despite having $m \in L^p(\Omega),$ $a_2$ will still be a bounded form on $H^1(\Omega) ,$ see proposition \ref{prop: bounded a_2} in Appendix \ref{App: Well posedness}.   
\begin{remark}
We can express relation \eqref{eq: weak form} equivalently as
\begin{align}
    (\mathcal S-k^2 \mathcal M_m -\imath k\mathcal B)u= f,
    \label{distibutional Helm}
\end{align}
as an $\overline{H^1(\Omega)}'$ distributional relation.
In particular,
\begin{align}
    \mathcal S: H^1(\Omega) \to \overline{H^1(\Omega)}'
\end{align}
acts as 
\begin{align}
    \langle\mathcal S y, \phi \rangle=(\nabla y ,\nabla \phi)_{L^2(\Omega;\C^2)},
\end{align}
and we define $\mathcal M_m$ and $ \mathcal B$ similarly as
\begin{align}
    \langle\mathcal B y, \phi \rangle=(y , \phi)_{L^{2}(\bd \Omega)},
\end{align}
and 
\begin{align}
     \langle\mathcal M_m y, \phi \rangle=(m y , \phi)_{L^2(\Omega)} .
\end{align}
It is clear that we can write 
\begin{align}
    \mathcal{S}-k^2\mathcal M_m-\imath k \mathcal B=\Phi \mathcal T (  I+k^2\mathcal A),
    \label{rel: equivalent writing}
\end{align}
where it follows that $\mathcal{S}-k^2\mathcal M_m-\imath k \mathcal B: H^1(\Omega)\to \overline{H^1(\Omega)'}$ is injective, invertible and has a bounded inverse, since $\Phi,  \mathcal T,    I+k^2\mathcal A$ are invertible with bounded inverses.
\label{rem: equivalent writing}
\end{remark}
It is worth noting (and particularly useful for the later analysis) that in the case of $q=0, f \in L^2(\Omega)$  we obtain by \cite{cummings2006} the following $H^2-$estimate.
\begin{proposition}
\label{prop: H^2 estimate}
    Assuming that $m=1, \ (q=0),$ and since  $\Omega$ is sufficiently smooth, we obtain that the solution of the Helmholtz boundary value problem
    \begin{align*}
        (-\Delta-k^2)u=f, \text{ in } \Omega,
    \end{align*}
    \begin{align*}
        (\bd_n-\imath k)u=0, \text{ on } \bd \Omega,
    \end{align*}
   obeys the following $H^2-$estimate,
    \begin{align}
        \|u\|_{H^2(\Omega)}\leq C\Big(k+1+\frac{1}{k}+\frac{1}{k^2}\Big) \|f\|_{L^2(\Omega)},
    \end{align}
    for a fixed $C>0.$
\end{proposition} 
\begin{remark}
\label{rem: H^2 bounded}
    The previous proposition essentially means that
    the operator $$(\mathcal S-k^2 \mathcal M_1-\imath k \mathcal B)^{-1} \mathcal J: L^2(\Omega)\to H^2(\Omega),$$ is bounded, with $\mathcal J:L^2(\Omega) \to  \overline{H^{1}(\Omega)}'$ being the usual embedding operator that acts as $\langle \mathcal J f,\phi \rangle = (f,\phi)_{L^2}, \ \phi \in H^1(\Omega).$
\end{remark}
\section{Main results \label{sec: main results}}
This section contains our main contributions.
In particular, in Section \ref{sub: vLS} we derive the variational Lippmann-Schwinger equation and we study the associated operators. In Section \ref{sub: q to u(q)} we show weak-to-strong sequential continuity of the parameter-to-wavefield map using the properties of the variational Lippmann-Schwinger operator. Finally in Section \ref{sub: case studies} we show existence of minimizers for two variants of the full waveform inversion method (reduced order model based and conventional)  for solving the inverse boundary value problem. 
\subsection{Variational Lippmann-Schwinger equation and its properties\label{sub: vLS}}
Here, we formulate the variational Lippmann-Schwinger equation (LS) equation and we show analytical properties of the variational LS operator. 
We saw in remark \ref{rem: equivalent writing} that the forward scattering problem can be written in the distributional sense as
\begin{align}
    (\mathcal S-k^2\mathcal M_m-\imath k \mathcal B)u= f, \text{ in } \cl{H^1(\Omega)}'
    \label{eq: distributional Helm}.
\end{align}
We apply the inverse operator $(\mathcal S-k^2\mathcal M_1-ik\mathcal B)^{-1}$  to relation \eqref{eq: distributional Helm} and we obtain equivalently
\begin{align}
    (\mathcal S-k^2\mathcal M_1-\imath k \mathcal B)^{-1}( \mathcal S-k^2\mathcal M_q-k^2 \mathcal M_1-\imath k \mathcal B)u=(\mathcal S-k^2\mathcal M_1-\imath k \mathcal B)^{-1}f, \text{ in } H^1(\Omega).
\end{align}
We define  
$\widetilde u_i:=(\mathcal S-k^2\mathcal M_1-ik \mathcal B)^{-1}f$ that can be roughly interpreted as a bounded domain analogue of an incident field. We thus obtain the relation
\begin{align}
    u-k^2 (\mathcal S-k^2\mathcal M_1-\imath k\mathcal B)^{-1}\mathcal M_qu=\widetilde u_i, \text{ in } {H^1(\Omega)}.
    \label{eq: initial Var LS}
\end{align}
We define the operator 
\begin{align}
     \widetilde {\mathcal V_q}:= (\mathcal S-k^2\mathcal M_1-\imath k\mathcal B)^{-1} \mathcal M_q: H^1(\Omega ) \to H^1(\Omega ),
\end{align}
and thus, given a generic right hand side $u_i \in H^1(\Omega )$, we obtain the variational Lippmann-Schwinger equation
\begin{align}\label{eq: VLS}
     u-k^2  \widetilde {\mathcal  V_q}u=u_i, \text{ in }  {H^1(\Omega)}.
\end{align}
\begin{remark}
In view of the classical Lippmann-Schwinger formulation of the forward scattering problem in $\R^d, \ d=2,3$, we can spot high-level similarities between our LS formulation and the classical one. 
First of all, as we saw in the introduction, in the classical case we have a volume potential operator that acts as
\begin{align}
    \mathcal V_q^{\text{classical}}(\psi):=\int_{\R^2} \Phi_k(x-y)q(y)\psi(y)dy=\{\Phi_k * (q\psi)\}(x), \ x\in \R^d, \psi \in L^2(\Omega),
\end{align}
with $\Phi_k$ being the fundamental solution of the unperturbed Helmholtz equation. 
The operator $\mathcal V_q^{\text{classical}}=\{\Phi_k * (q \ \cdot)\}$ can be thought as the inverse operator of the unperturbed Helmholtz operator, which can be interpreted roughly as
\begin{align*}
    \Phi_k*(q\psi)=(-\Delta-k^2)^{-1} (q \psi ).
\end{align*}
In our case, we have a distributional operator 
$\mathcal S-k^2\mathcal M_1 -ik \mathcal B$ which is essentially the bounded domain analogue of $(-\Delta-k^2)$ with Sommerfeld radiation condition.
\end{remark}
The well‑posedness of the operator $\widetilde{\mathcal V_q}$ is summarized in the following theorem. Its proof relies on several auxiliary lemmas that we show  immediately afterwards.
\begin{theorem}
   Suppose assumption \ref{assumption} holds. Then, the operator 
   \begin{align}
       \widetilde  {\mathcal V_q}=(\mathcal S-k^2\mathcal M_1-
    \imath k\mathcal B)^{-1}\mathcal M_q: H^1(\Omega) \to H^1(\Omega)
   \end{align}
is smoothing in the sense 
\begin{align}  \widetilde  {\mathcal V_q}: H^1(\Omega)\to H^2(\Omega) \subspace H^1(\Omega),
    \end{align}
    and therefore the operator 
    \begin{align}
        \mathcal V_q:= id_{H^2\to H^1} \circ \widetilde  {\mathcal V_q}: H^1(\Omega) \to H^1(\Omega)
    \end{align}
    is compact.
    Also, the variational LS operator \begin{align}
    I-k^2\mathcal V_q:H^1(\Omega) \to H^1(\Omega) \end{align}
    is invertible with a bounded inverse.  
    \label{th: well posed V_q}
\end{theorem}
We obtain directly the following result.
\begin{corollary}
\label{coro: equivalent vLS-Helm}
    Assume $k>0, f\in \overline{H^1(\Omega)'}$ and $q$ as in assumption \ref{assumption}. Then, if $u \in H^1(\Omega)$ solves uniquely
    \begin{align}
        u-k^2 \mathcal V_qu=(\mathcal S-k^2\mathcal M_1-
    \imath k\mathcal B)^{-1} f, \ {H^1(\Omega)},
    \end{align}
    then equivalently, $u$ solves equation \eqref{distibutional Helm} uniquely.
\end{corollary}
Before we dive deeper into the more technical aspects, let us break down more the action of $\mathcal V_q$. Given an element $g\in H^1(\Omega)$, then ${\mathcal V_q} g=\widetilde{\mathcal V_q} g \in H^1(\Omega)$ and the equation 
\begin{align}
    \widetilde{\mathcal V_q} g=(\mathcal S-k^2\mathcal M_1-\imath k\mathcal B)^{-1}\mathcal M_qg=w,
\end{align}
is equivalent to solving the following Helmholtz problem,
\begin{align}
    (-\Delta-k^2)w=qg, \text{ in } \Omega, \label{rel: aux Helm}
\end{align}
\begin{align}
    (\bd_n-\imath k)w=0,  \text{ in } \bd \Omega.  \label{rel: aux BC}
\end{align}
%The following proposition guarantees that the right hand side, $gq,$ is square integrable, and thus that operator $V_q$ is well defined. 
\begin{remark}
     Notice that we associate $\mathcal M_qg\in \overline{H^{1}(\Omega)}'$ with the product $qu$ through
     \begin{align}
         \langle \mathcal  M_q g,v\rangle = \langle \mathcal J(\widetilde {\mathcal M_q}g),v\rangle =\int_\Omega (qg)\overline v dx= (\widetilde {\mathcal M_q}g,v)_{L^2},
     \end{align}
    with $\mathcal J:L^2(\Omega) \to  \overline{H^{1}(\Omega)}'$ being the usual embedding operator, and
    \begin{align*}
        \widetilde {\mathcal M_q}: H^1(\Omega)\to L^2(\Omega)
    \end{align*}
    being the multiplication operator defined by $q$. % (which is bounded as we will see).
    Essentially, if $\langle \mathcal  M_q g, \cdot \rangle $  is used as a right hand side of equation \eqref{eq: distributional Helm} then it introduces the volume source term $qg.$
\end{remark}
\begin{lemma}
\label{lemma: M_q bounded}
    Given $q$ according to assumption \ref{assumption}, then $\widetilde {\mathcal M_q}: H^1(\Omega) \to L^2(\Omega)$ is bounded, and thus
    \begin{align}
        \widetilde{\mathcal V_q}:H^1(\Omega)\to H^1(\Omega)
    \end{align}
    is bounded.
\end{lemma}
\begin{proof}
    First, we treat the case when $p\in(2,\infty)$. Since $u\in H^1(\Omega) \subset L^s(\Omega), \ s\geq  2$, we can find an $s$ with 
    \begin{align}
        \| qu \|_{L^{sp/(p+s)}}\leq \|q\|_{L^p} \|u\|_{L^s},
        \label{rel: Generalise Holder}
    \end{align}
    using the generalized Hölder inequality.
    We want $sp/(s+p)=2$, thus we select the value $s=2p/(p-2)$.
    Therefore we obtain 
    \begin{align}
        \|\widetilde {\mathcal M_q}u \|_{L^2}=\| qu\|_{L^2}\leq \|q\|_{L^p}\|u\|_{L^s}\leq  C_{1}\|q\|_{L^p}\|u\|_{H^1},
    \end{align}
    with $C_1$ being the bound of the embeding operator $id:{H^1(\Omega)\to L^s(\Omega)}$.
   We see that 
   \begin{align}
       |\langle \mathcal{M}_q u,v\rangle| =|(\widetilde {\mathcal M_q}u,v)_{L^2}|\leq C_{1}\|q\|_{L^p}\|u\|_{H^1} \|v\|_{H^1},
       \end{align}
       therefore we obtain
       \begin{align}
       \|\mathcal{M}_q u\|_{\overline{H^1}'} \leq C_{1}\|q\|_{L^p} \|u\|_{H^1},
   \end{align}
    which means that the operator $\mathcal M_q$ is bounded.
    
    It is also clear that $\widetilde {\mathcal V_q}: H^1(\Omega)  \to  H^1(\Omega) $ is bounded as a composition of $\mathcal M_q$ with $(\mathcal S-k^2 \mathcal M_1 -\imath k \mathcal B)^{-1}$, which is also bounded as we discuss in remark \ref{rem: equivalent writing}. 
    
    Finally, when $p=\infty$, we obtain $|q(x)u(x)|^2 \leq \|q\|_{L^\infty}^2 |u(x)|^2,  $ for almost all $x \in \Omega.$ Thus we obtain that $\|qu\|_{L^2}\leq  \|q\|_{L^\infty}\|u\|_{L^2}$ and we can continue as in case $p\in(2,\infty).$
\end{proof}
According to proposition \ref{prop: H^2 estimate}, since ${\mathcal M_q}g$ corresponds to a volume source $\widetilde {\mathcal M_q}g=qg\in L^2(\Omega)$ means  that 
$\widetilde{\mathcal V_q} g=w \in H^2(\Omega) $. Therefore, we can take
\begin{align}
    \mathcal V_q=id_{H^2\to H^1} \circ \widetilde{\mathcal V_q}
    %=id_{H^2\to H^1} (\mathcal S-k^2\mathcal M_1-\imath kB)^{-1}{\mathcal M_q}
    : H^1(\Omega )\to H^2 (\Omega )\stackrel{c}{\subspace} H^1 (\Omega).
    \label{rel: Vq in H^2}
\end{align}
This leads to the compactness of $\mathcal V_q$.
\begin{lemma}
\label{lemma: compact V_q}
    Given $q$ and $\Omega$ according to assumption \ref{assumption}, then 
    \begin{align}
        \mathcal  V_q:H^1(\Omega)\to H^1(\Omega)
    \end{align}
    is a compact operator.
\end{lemma}
\begin{remark}
    In view of the two above lemmas, the reason why we assume $p>2$ in assumption \ref{assumption} becomes clear. To obtain $qg\in L^2(\Omega),$ and therefore obtain a right hand side for \eqref{rel: aux Helm} that yields an $H^2-$regular solution, we need to assume that $p\in(2,\infty].$
\end{remark}
We would like to use the Riesz-Fredholm theory for showing invertibility of the variational LS equation. Since we have shown that $\mathcal V_q$ is compact, we now need to show that $I-k^2\mathcal V_q$ is injective.
\begin{lemma}
\label{lemma: injective V_q}
    Given $q$ and $\Omega$ according to assumption \ref{assumption}, the operetor
    \begin{align}
        I-k^2\mathcal V_q: H^1(\Omega)\to H^1(\Omega)
    \end{align}
    is injective.
\end{lemma}
\begin{proof}
    Consider the variational LS equation with zero right hand side,
    \begin{align}
        (I-k^2\mathcal V_q)u=0 \text{ in } H^1(\Omega).
        \nonumber
    \end{align}
    Since relations \eqref{eq: initial Var LS} and \eqref{eq: distributional Helm} are equivalent, we obtain that $u$ should also satisfy
     \begin{align}
    (\mathcal S-k^2\mathcal M_{1+q} -\imath k\mathcal B)u=0, \text{ in } \overline{H^1(\Omega)}'.
    \end{align}
    As we mentioned in remark \ref{rem: equivalent writing}, the variational Helmholtz operator is injective, thus we obtain that $u=0,$ therefore
    \begin{align}
        (I-k^2\mathcal V_q) : H^1(\Omega) \to H^1(\Omega)
    \end{align}
    is also injective.
\end{proof}
We now collect all the above results to prove the main result of this subsection, Theorem \ref{th: well posed V_q}.
\begin{proof}[Proof of Theorem \ref{th: well posed V_q}.]
    We have shown in Lemma \ref{lemma: compact V_q} that $\mathcal V_q: H^1(\Omega)\to H^1(\Omega)$ is compact and in Lemma \ref{lemma: injective V_q} that $I+k^2\mathcal V_q: H^1(\Omega)\to H^1(\Omega)$ is injective. Thus, we can use  the Riesz Fredholm theory to conclude that 
    $I+k^2\mathcal V_q: H^1(\Omega)\to H^1(\Omega)$ is invertible with bounded inverse. Therefore, the variational LS equation \eqref{eq: VLS} has a unique solution. 
\end{proof}
\subsection{Properties of the Variational Lippmann-Schwinger operator and parameter-to-state map\label{sub: q to u(q)}}
In this subsection we show analytical properties of the parameter-to-state map $$q\mapsto u(q).$$ 
%% Ask the robot to rephrase
%We base our proofs of the analytical properties of the variational LS operator by working in the  same spirit as in \cite{lechleiter2013tikhonov}. Particularly, we use the theory of collectively compact operators for our proposed equation of Lippmann-Schwinger type. 
Our proofs of the analytical properties of the variational LS operator follow a similar approach as in \cite{lechleiter2013tikhonov}. In particular, we make use of the theory of collectively compact operators in the context of our proposed Lippmann–Schwinger type equation.

In the following theorem we formulate the main result of this subsection. Between the statement of the theorem and its proof we show a number of auxiliary results in the form of lemmas needed for the proof of the theorem.
\begin{theorem}
\label{th: q mapsto u(q)}
Suppose the assumption \ref{assumption} and take a sequence $\{q_n\}_{n \in \N} \subset L^p(\Omega; [-1,\infty)),$ and $q\in L^p(\Omega; [-1,\infty))$  for some$ \ p\in (2,\infty]$ with 
    \begin{align}
        q_n \weakto q \text{ in }\sigma(L^p(\Omega),(L^{p}(\Omega))'), \text { as } n\to \infty, \text{ if } p\in(2,\infty),
    \end{align}
    or 
    \begin{align}
        q_n \stackrel{*}\weakto q \text{ in }\sigma(L^\infty(\Omega),L^{1}(\Omega)), \text { as } n\to \infty.
    \end{align}
    The parameter-to-state map is weak to strong continuous in the sense
    \begin{align}    q_n\stackrel{\sigma(L^p,(L^{p})') }\weakto q \Ra u(q_n)\stackrel{H^1}{\to}u(q),  \text{ if } p\in(2,\infty)
     \end{align}
     or %% Add?
     \begin{align}    q_n\stackrel{\stackrel{\sigma(L^\infty,L^{1})}{*} }\weakto q \Ra u(q_n)\stackrel{H^1}{\to}u(q), \ \text{ if } p=\infty,
     \end{align}
     as $n\to \infty,$ 
     where $u(q_n)$ solves the Helmholtz problem \eqref{eq: distributional Helm} with $m=1+q_n$, $n\in \N$ and similarly $u(q)$ solves  the same problem with $m=1+q$.
\end{theorem}
We postpone the proof of Theorem \ref{th: q mapsto u(q)} to the end of this subsection after showing a number of auxiliary results.
For the sake of completeness we first state the following proposition.
\begin{proposition}
\label{prop: q_n to q means q geq 0}
    Given a sequence $\{q_n\}_{n\in \N} \subset L^p(\Omega; [-1,\infty)), \ p\in(2,\infty]$ with 
    \begin{align}
        q_n \weakto q \text{ in }\sigma(L^p(\Omega),(L^{p}(\Omega))'), \text { as } n\to \infty, \text{ if } p\in(2,\infty),
    \end{align}
    or 
    \begin{align}
        q_n \stackrel{*}\weakto q \text{ in }\sigma(L^\infty(\Omega),L^{1}(\Omega)), \text { as } n\to \infty,
    \end{align}
    then $q\geq -1$ almost everywhere.
\end{proposition}
\begin{proof}
    For an almost similar result refer \cite{lechleiter2013tikhonov}. Appendix \ref{App: Extra} contains the proof of this proposition
\end{proof}
One key concept that is going to help us show analytical properties of the parameter-to-state map is the fact that the variational LS operator can yield a sequence of collectively compact operators. 
\begin{lemma}
    Given a sequence $\{q_n\}_{n\in \N}\subset L^p(\Omega;[-1,\infty))$ and $q\in L^p(\Omega)$ with  $q_n \weakto q$ in the weak topology of $L^p_{}(\Omega), \ p\in(2,\infty)$, or $q_n \stackrel{*}\weakto q$  in the weak star topology $\sigma(L^\infty(\Omega),L^1{}(\Omega))$ when $p=\infty$, then the sequence
    \begin{align}
        \{\mathcal V_{q_n}\}_{n\in \N}:H^1(\Omega) \to H^1(\Omega)
    \end{align}
    is a sequence of  collectively compact operators.
\end{lemma}
\begin{proof}
    We need to show that given any bounded set, $U$, in $H^1(\Omega)$ then,
    \begin{align}
        S=\{\mathcal V_{q_n} u: n\in \N, u\in U\}
    \end{align}
    is relatively compact in $H^1(\Omega).$
    As we noted in relation \eqref{rel: Vq in H^2}, the image of $\mathcal V_q$ is included in $H^2(\Omega)$. %, thus Since $S \subset H^2(\Omega)$. %$\text{Range}(\mathcal  
    % V_{q_n})=H^2(\Omega)$ %% ADD?
    % ),
    % we can view the set $S$ in $H^1(\Omega)$ through the compact embedding $id:H^2(\Omega)\to H^1(\Omega)$ as
    %\begin{align}
     %   S= id_{H^2\to H^1}  S.
    %\end{align}
     Next, we observe that the set $S$ is bounded in $H^2(\Omega).$ To see that  we first fix, $n\in \N$ and $u
    \in H^1(\Omega
    )$. We observe that
    \begin{align}
        \mathcal V_{q_n} u = (\mathcal S-k^2\mathcal  M_1-\imath k\mathcal B)^{-1}\mathcal M_{q_n} u=w_n,
    \end{align}
    with $w_n$ satisfying
    \begin{align}
    (-\Delta-k^2)w_n=q_nu, \text{ in } \Omega, \nonumber
\end{align}
\begin{align}
    (\bd_n-\imath k)w_n=0  \text{ in } \bd \Omega.
\end{align}
We know from proposition \ref{prop: H^2 estimate}, that
\begin{align}
    \|w_n\|_{H^2} \leq C \big(1+\frac{1}{k}+\frac{1}{k^2}\big) \| q_nu \|_{L^2}\leq C \big(1+\frac{1}{k}+\frac{1}{k^2}\big) \| q_n\|_{L^p}\|u \|_{L^s},
\end{align}
as we can choose $s=2p/(p-2)$, such that the generalized Hölder inequality holds (relation \eqref{rel: Generalise Holder}), or if $p=\infty$ we choose $s=2$. 
Thus, for all $n\in \N$ and $ u\in H^1(\Omega)$ we obtain
\begin{align}
   \|\mathcal V_{q_n} u\|_{H^2} =\|w_n\|_{H^2} \leq C \big(1+\frac{1}{k}+\frac{1}{k^2}\big) \| q_n\|_{L^p}\|u \|_{L^s} \leq  C \big(1+\frac{1}{k}+\frac{1}{k^2}\big) B_u B_q,
\end{align}
where $B_u$ is the norm bound of the set $U$ and $B_q$ is the norm bound of the weakly (or weakly-$*$ if $p=\infty$) converging sequence $\{q_n\}_{n\in \N}$. This means that $S$ is bounded in $H^2(\Omega)$ and thus relatively compact in $H^1(\Omega).$
\end{proof}
For the later analysis we need the following auxiliary result.% of the above lemma, we first need the following result. (make it a prop and put above)
\begin{lemma}
\label{lemma: Mq_n pointwise}
      Given a sequence $\{q_n\}_{n\in \N}\subset L^p(\Omega;[-1,\infty))$ and $q\in L^p(\Omega)$ with  $q_n \weakto q$ in the weak topology of $L^p_{}(\Omega), \ p\in(2,\infty)$, or $q_n \stackrel{*}\weakto q$  in the weak star topology $\sigma(L^\infty(\Omega),L^1{}(\Omega))$ when $p=\infty$, then   
      \begin{align}
          \widetilde {\mathcal   M_{q_n}} u\weakto \widetilde { \mathcal M_q}  u \text{ in } \sigma(L^2(\Omega),L^2(\Omega)),
      \end{align}
     pointwise $\fa u \in H^1(\Omega). $
\end{lemma}
\begin{proof}
    We first study the case that $p\in(2,\infty).$ Given a fixed $u\in H^1(\Omega)$, the operator $ \widetilde {\mathcal M_u}q \ =\widetilde{\mathcal M_q }u=qu: L^p (\Omega)\to L^2(\Omega) $ is bounded as we saw in the proof of Lemma \ref{lemma: M_q bounded}. Therefore, if we take the weakly convergent sequence  
    \begin{align}
        q_n\weakto q \text{ in } \sigma(L^p(\Omega),(L^p(\Omega))'), \ p\in(2,\infty)
    \end{align}
    we obtain
    \begin{align}
        u q_n \weakto uq, \ \text{ in } \sigma(L^2(\Omega),L^2(\Omega)).
    \end{align}
    Thus, we obtain the limit
    \begin{align}
       \widetilde{\mathcal M_{q_n} }u\weakto \widetilde{ \mathcal M_q} u \text{ in } \sigma(L^2(\Omega),L^2(\Omega)), \ \fa u \in H^1(\Omega).
        \label{rel: lim Mqnu}
    \end{align}
    Now, when $p=\infty$ we assume 
    \begin{align}
         q_n\stackrel{*}\weakto q \text{ in } \sigma(L^\infty(\Omega),L^1(\Omega)),
    \end{align}
    which is equivalent to $\int_\Omega fq_n dx\to \int_\Omega fqdx, $ for any $ f\in L^1(\Omega). $ 
    We fix $u\in H^1(\Omega)$ and take a $v\in L^2(\Omega).$ Then 
    \begin{align}
         (uq_n-uq,v  )_ {L^2}=\int_\Omega \overline{v} (uq_n-uq)dx= \int_\Omega \overline{v}u (q_n-q) dx, \fa n\in \N.
    \end{align}
    Since $\overline{v}u \in L^1(\Omega)$ (Cauchy-Schwarz), and since $q_n \stackrel{*}{\weakto}$q we obtain that $(uq_n-uq,v  )_ {L^2}\to 0, \ n\to \infty,$ thus we also obtain \eqref{rel: lim Mqnu} in case $p=\infty.$
\end{proof}
The next result guarantees the pointwise convergence of the the sequence $(\mathcal V_{q_n})$, given a weakly convergent sequence, $\{q_n\}_{n\in \N}$.
\begin{lemma}
Given a sequence $\{q_n\}_{n\in \N}\subset L^p(\Omega;[-1,\infty))$ and $q\in L^p(\Omega)$ with  $q_n \weakto q$ in the weak topology of $L^p_{}(\Omega) $ when $ p\in(2,\infty)$, or $q_n \stackrel{*}\weakto q$  in the weak star topology $\sigma(L^\infty(\Omega),L^1{}(\Omega))$ when $p=\infty$, then  
\begin{align}
    \mathcal V_{q_n} u \stackrel{H^1}{\to} \mathcal V_{q} u, \text { as } n\to \infty,
\end{align}
for all $u$ in $H^1(\Omega)$.
\label{lemma: Vq_nu pointwise}
\end{lemma}
\begin{proof}
    We fix $u\in H^1(\Omega)$ and using Lemma \ref{lemma: Mq_n pointwise} we obtain that
\begin{align}
        q_n \stackrel{\sigma(L^p,(L^p)')}{\weakto} q \Ra \widetilde{\mathcal M_{q_n} }u \stackrel{\sigma(L^2,L^2)}{\weakto} \widetilde{\mathcal M_{q} } u, \ p \in(2,\infty),
        \end{align}
        or
        \begin{align}
        q_n \stackrel{\stackrel{\sigma(L^\infty,L^1)}{*}}{\weakto} q \Ra \widetilde{\mathcal M_{q_n} }u \stackrel{\sigma(L^2,L^2)}{\weakto} \widetilde{\mathcal M_{q} }u, \ p=\infty,
        \end{align}
        as $n\to \infty.$
    The inverse operator of the constant coefficient Helmholtz operator is compact when we use volume sources that belong in $L^2(\Omega)$ (see Proposition \ref{prop: H^2 estimate} and remark \ref{rem: H^2 bounded}), 
\begin{align}
    id_{H^2\to H^1} \circ (\mathcal S-k^2\mathcal M_1 -\imath k \mathcal B)^{-1} \mathcal J :L^2(\Omega)\to H^2(\Omega) \stackrel{c}{\subspace} H^1(\Omega),
    \nonumber
\end{align}
with $\mathcal J:L^2(\Omega)\to \overline{H^1(\Omega)}' $ being the canonical embedding and $\mathcal J \widetilde{\mathcal M_{q} }=\mathcal M_{q}.$ Therefore we obtain the following strong convergence result
        \begin{align}
          id_{H^2\to H^1} \circ  (\mathcal S-k^2\mathcal M_1 -\imath k \mathcal B)^{-1} \mathcal  M_{q_n} u \stackrel{H^1}{\to}id_{H^2\to H^1} \circ  (\mathcal S-k^2\mathcal M_1 -\imath k \mathcal B)^{-1} \mathcal M_{q} u,
    \end{align}
    as $n\to\infty.$
\end{proof}
We are now ready to prove the main result of this subsection using all the results above.  
\begin{proof}[Proof of Theorem \ref{th: q mapsto u(q)}]
Since we want to treat the distributional form of the Helmhlotz equation \eqref{distibutional Helm}, we once again denote $u_i=(\mathcal S-k^2\mathcal M_1-ik\mathcal B)^{-1} f.$
    Therefore, as we saw in corollary \ref{coro: equivalent vLS-Helm}, instead of the distributional form of the Helmhlotz equation we can equivalently consider the variational LS equations
    \begin{align}
        u-k^2\mathcal V_q u=u_i ,
    \end{align}
     \begin{align}
        u_n-k^2\mathcal V_{q_n} u_n=u_i,
    \end{align}
    that hold in $H^1(\Omega)$ for $n\in \N$ assuming the given weak (or weak$-*$) converging sequence $\{q_n\}_{n\in \N}$ and the limit $q$. According to Theorem \ref{th: collectively compact theorem}, we observe that   $k^2 \mathcal V_q$ is compact and $(I-k^2\mathcal V_q)$ is injective. Also,  $\{\mathcal V_{q_n}\}_{n\in \N}$ is a sequence of collectively compact operators that strongly converges  pointwise as
    \begin{align}
        \mathcal V_{q_n}u\stackrel{H^1}\to  \mathcal V_qu, \ n\to \infty,
    \end{align}
    for $u\in H^1(\Omega).$
     Then, using inequality \eqref{rel: colectively compact estimate} found in Corollary  \eqref{cor: colectively compact estimate} of Appendix \ref{App: Collectively Compact Op}, we obtain the estimate
     \begin{align}
         \|u_n-u\|_{H^1}\leq Ck^2\| \mathcal V_{q_n}u-\mathcal V_qu\|_{H^1}.
     \end{align}
     Consequently we obtain the weak-to-strong sequential continuity of the parameter-to-wavefield map,
     \begin{align}
         q_n\stackrel{\sigma(L^p,(L^p)')}{\weakto} q \Ra u_n\stackrel{H^1}{\to}u, \ p\in(2,\infty),
     \end{align}
     or
      \begin{align}
         q_n\stackrel{\stackrel{\sigma(L^\infty,L^1)}{*}}{\weakto} q \Ra u_n\stackrel{H^1}{\to}u, \ p=\infty.
     \end{align}
\end{proof}
\subsection{Case-studies\label{sub: case studies}}

In this section, we use the weak-to-strong continuity of the parameter-to-wavefield map, as seen in Theorem \ref{th: q mapsto u(q)}, for showing existence of minimizers for two popular optimization methods, aimed at solving the Helmholtz inverse boundary value problem, namely reduced order model (ROM)-based inversion and conventional full waveform inversion (FWI). As we shall see, the continuity result is particularly important in the ROM setting, where the misfit functionals depend on volume inner products of the wavefields. 
\subsubsection{Data-driven reduced order model based inversion \label{sub: ROM based FWI}}
Recently, reduced order models (ROMs) corresponding to scattering problems on bounded domains have been employed for developing new waveform inversion techniques, see for example
%\textcolor{red}{add more citations}
\cite{tataris2025inversescatteringschrodingerequation, tataris2023reduced,borcea2023data,borcea2025reduced}. Below we summarize how we can obtain a data driven ROM from receiver and boundary data. The first step is to consider the wavenumbers
\begin{align*}k_1<k_2<\cdots<k_{N}
\end{align*}
that yield the wavefields,
\begin{align}
    u^{(s)}_1, u^{(s)}_2, \cdots u^{(s)}_{N},
\end{align}
that solve the Helmholtz problem 
\eqref{eq: distributional Helm} with given sources indexed by $s=1,...,M,$ 
corresponding to an underlying $m=1+q$ that needs to be reconstructed.

We consider measurements/data consisting of boundary traces 
\begin{align}
    \mathcal{D}=
    \{ u_i^{(s)}|_{\bd \Omega}, \bd_k u_i^{(s)}|_{\bd \Omega}\}_{i=1,...,{N}, \ s=1,...,M},
\end{align} 
that also include the traces of Frechet derivatives of the wavefields with respect to the wavenumbers.
We also consider the receiver responses,
\begin{align}
    \mathcal{E}= \{ \mathcal E^{(r,s)}_i=\langle f^{(r)}, u_i^{(s)}\rangle|\}_{i=1,...,{N}, \ r,s=1,...,M}.
\end{align}
We define the ROM matrices, in particular the \textit{mass}, \textit{stiffness} and \textit{boundary} matrices,  
$\mathbf M, \mathbf S, \mathbf B$ $\in$ $\C^{NM \times NM } $ respectively.
Due to the indexing of snapshots according to sampling
wavenumbers and source numbers, all three matrices ROM matrices have a block structure
consisting of $N \times N$ blocks of size $M \times M$ each, for example,
\begin{equation}
\mathbf S = \begin{bmatrix}
\mathbf S_{11} & \mathbf S_{12} & \ldots & \mathbf S_{1N} \\ 
\mathbf S_{21} & \mathbf S_{22} & \ldots & \mathbf S_{2N} \\
\vdots & \vdots & \ddots & \vdots \\
\mathbf S_{NN} & \mathbf S_{NN} & \ldots & \mathbf S_{NN} \\
\end{bmatrix} \in \C^{NM \times NM}.
\label{eqn:mats}
\end{equation}
The elements of the block matrices include inner products of the wavefields as
\begin{align}
    [\mathbf M_{ij}]_{rs}=  \int_\Omega (1+q) u_i^{(s)} \overline{u_j^{(r)}}dx,   \ s,r=1,...,M,
\end{align}
\begin{align}
    [\mathbf S_{ij}]_{rs}= \int_\Omega \nabla u_i^{(s)} \overline{\nabla u_j^{(r)}}dx,   \ s,r=1,...,M,
    \label{rel: stiffness elements}
\end{align}
and 
\begin{align}
    [\mathbf B_{ij}]_{rs}= \int_\Omega  u_i^{(s)} \overline{u_j^{(r)}}d\Sigma(x) , \ s,r=1,...,M,
\end{align}
for $i,j=1,...,N.$
%The terms involving the true wavefield are reconstructed exactly using boundary data/traces of the wavefields.
The key property of ROMs is that it is possible to compute exactly the entries of the matrices $\mathbf{M,S}$ using the entries of the matrix $\mathbf B$ and the data $\mathcal D$  and $\mathcal{ E}$, see \cite{TristanAndreas23ROM, tataris2023reduced, tataris2025inversescatteringschrodingerequation} 

One variant of the ROM based FWI can be defined using a misfit functional that is based on volume inner products of the wavefields via the elements of the ROM stiffness matrix. Specifically, for a trial parameter $q$, 
we define the functional 
\begin{align*}
    \phi_{ROM}(q):=\frac{1}{2}\|\mathbf S-\mathbf{S}(q)\|_F^2,
\end{align*}
where $\|\cdot\|_F$ denoting the Frobenious norm and $\mathbf{S}(q)$ is the stiffness matrix corresponding to a trial parameter $q$.  Equivalently, the misfit functional can be defined via relation \eqref{rel: stiffness elements} as well. Therefore, we can also write 
\begin{align}
    \phi_{ROM}(q) :=\frac{1}{2} \sum_{i,j=1}^N\sum_{r,s=1}^M \phi_{ijrs}(q),
\end{align}
where
\begin{align}
    \nonumber
    \phi_{ijrs}(q)=\Big|\int_\Omega \nabla u_i^{(s)}(q) \overline{\nabla  u_j^{(r)}(q)}dx-\int_\Omega \nabla  u^{(s)}_i \overline{\nabla  u^{(r)}_j}dx \Big|^2, %:=
\end{align}
where $u_i^{(s)}(q)$ solves the Helmholtz problem \eqref{eq: distributional Helm} at wavenumber $k_i$ corresponding to a source indexed by $s$. 
As has been observed in the one dimensional case \cite{tataris2023reduced}, such a ROM based functional can pose challenges when it comes to showing existense of minimizers when using regularization. In particular, the problem comes from the inner products between wavefields and their behaviour under weak convergence. 
However, as we will demonstrate, Theorem \ref{th: q mapsto u(q)} plays a key role in overcoming these difficulties and in establishing the existence of minimizers.
To do that, we first need to
define the admissible set
\begin{align}
    \mathbb{Q}_{ad}=\{q \in L^p(\Omega): q\geq -1 \}.
\end{align}
This set is convex, weakly closed when $p\in(2,\infty)$ or weakly$-*$ closed when $p=\infty$ as was seen in proposition \ref{prop: q_n to q means q geq 0}. %thus closed in the norm topology of $L^p(\Omega;\R)$.
The following theorem demonstrates the use of Theorem \ref{th: q mapsto u(q)} and guarantees the existence of local minimizers of the  ROM based inversion functional under $L^p-$norm regularization.
\begin{theorem}
\label{th: existence for ROM FWI}
    Given a regularization parameter $a>0$, the regularized  ROM based inversion functional
    \begin{align}
        \phi_{\text{ROM}}(q;a):=\phi_{\text{ROM}}(q)+a \|q\|_{L^p(\Omega)}
    \end{align}
    admits a minimizer in $\mathbb Q_{ad}.$
\end{theorem}
\begin{proof}
    Since $\phi_{\text{ROM}}(q;a)\geq 0$ for all $q \in \mathbb Q_{\text{ad}}$, there exists $\theta>0$ such that
\begin{align}
    \theta =\inf_{q \in \mathbb Q_{\text{ad}}}\phi_{\text{ROM}}(q;a).
\end{align}
Therefore, there is a sequence that converges to $\theta$, let that sequence be
\begin{align}
    \{\phi_{\text{ROM}}(q_n;a)\}_{n\in \mathbb{N}}  \subset \{ \phi_{\text{ROM}}(q;a) : q\in \mathbb Q_{\text{ad}}\}.
\end{align} 
Since we use regularization, and the sequence $\{\phi_{\text{ROM}}(q_n;a)\}_{n\in \mathbb{N}}$ is included in a bounded interval, then the sequence $\{q_n\}_{n\in \mathbb{N}}\subset \mathbb Q_{ad}$ is bounded in $L^p(\Omega)$. Therefore, there is a subsequence 
$\{q_\nu\}_{\nu\in N_1} \subset \{q_n\}_{n\in \mathbb{N}}$ that has a weak limit, let $\widehat q$, in
$\sigma(L^p(\Omega),{L^p(\Omega)}')$ or in $\sigma(L^\infty (\Omega),L^1(\Omega))$ if $p=\infty$.
As we saw in Proposition \ref{prop: q_n to q means q geq 0}, the limit element $\widehat q$ is included in $\mathbb{Q}_{ad}$ since the set is weakly closed when $p\in(2,\infty),$ and closed in the weak star topology $\sigma(L^\infty(\Omega),L^1(\Omega))$ when $p=\infty.$
Finally, by Theorem \ref{th: q mapsto u(q)} it follows that we have the  following strong limit
\begin{align}
    u_i^{(s)}(q_\nu) \to u^{(s)}_i(q),  \text{  in } H^1(\Omega), \text{ as } \nu\to \infty, \  \nu \in N_1,
\end{align}
for $i=1,...,N, \ s=1,...,M$.
%Furthermore, $Q$ is closed and convex, therefore weakly closed, so we also obtain that $m^\star \in Q$.
Moreover, since $\nabla :H^1(\Omega)\to L^2(\Omega;\C^2)$ is bounded, we obtain that
\begin{align}
  \nabla u_i^{(s)}(q_\nu )\to \nabla u_i^{(s)}(q) , \ \text{ in } L^2(\Omega), \text{ as } \nu\to \infty, \  \nu \in N_1,
\end{align}
for all $i=1,...,N, \ s=1,...,M.$ 
Finally, for all $i,j=1,...,N, \ r,s=1,...,M,$ we obtain for the $(i,j,r,s)-$component of the Frobenius norm sum of the data fit term
\begin{align*}
     \lim_{\nu \in N_1} \phi_{ijrs}(q_\nu) :=\lim_{\nu \in N_1}\Big|\int_\Omega \nabla u_{i}^{(s)}(q_\nu) \cdot\overline{\nabla{u_j^{(r)}(q_\nu)}}dx-S_{ij}^{obs.}\Big|^2=\\
    \Big|\int_\Omega \nabla u_{i}^{(s)}(q) \cdot\overline{\nabla{u_j^{(r)}(q)}}dx-S_{ij}^{obs.}\Big|^2=\phi_{ijrs}(\widehat q).
\end{align*}
Therefore, by the strong convergence of the sequences $\{\phi_{ijrs}(m_\nu)\}_{\nu \in N_1}, \ i,j=1,...,N, r,s=1,...,M$ and the weak lower semicontinuity of the norm, we obtain
\begin{align}
    \theta=\lim_{\nu \in N_1 } \phi_{\text{ROM}}(q_\nu;a) 
    \geq 
    \liminf_{\nu \in N_1} \phi_{\text{ROM}}(q_\nu;a)  = \\ \nonumber \sum_{ij=1}^{N}\sum_{ij=1}^{M}\phi_{ij}(\widehat q)+ a\liminf_{\nu \in N_1} \| q_\nu\|_{L^p(\Omega)} \geq   \phi_{\text{ROM}}(\widehat q)\geq \theta,
\end{align}
which means that the limit element $\widehat q\in \mathbb Q_{ad}$ is a local minimizer of $\phi_{ROM}.$
\end{proof}
\subsubsection{Data misfit conventional FWI}
We use a similar setting as in the previous Section \ref{sub: ROM based FWI}. We use the data set $\mathcal E$ and assume that the sources act as receivers as well. We define the conventional data misfit FWI functional
\begin{align}
    \phi_{\text{FWI}}(q)=\frac{1}{2}\sum_{r,s=1}^{M}\sum_{i=1}^{N}|\langle f^{(r)}, u^{(s)}_i(q) \rangle- \mathcal E^{(r,s)}_i|^2, \quad q\in \mathbb Q_{ad},
\end{align}
that measures the mismatch between observed and modeled data.
We obtain similarly as in the ROM based functional case of the previous section, that there exist local minimizers for $\phi_{\text{FWI}}$ in $\mathbb Q_{ad}.$

\begin{theorem}
    Given a regularization parameter $a>0$, the regularized FWI functional
    \begin{align}
        \phi_{\text{FWI}}(q;a):=\phi_{\text{FWI}}(q)+a \|q\|_{L^p(\Omega)}
    \end{align}
    admits a minimizer in $\mathbb Q_{ad}.$
\end{theorem}
\begin{proof}
    We begin the proof similarly as the proof of Theorem \ref{th: existence for ROM FWI} that describes the ROM based case.
    First we denote the infimum of the functional on the admissible set as
    \begin{align}
    \zeta =\inf_{q \in \mathbb Q_{\text{ad}}}\phi_{\text{FWI}}(q;a).
\end{align}
Similarly to the ROM based case, since we are using a regularization term, there is a minimizing sequence $ \{q_n\}_{n\in \mathbb{N}}$ that is bounded in $L^p(\Omega)$. Consequently, we can extract a subsequence 
$\{q_\nu\}_{\nu\in N_1} \subset \{q_n\}_{n\in \mathbb{N}}$ that converges weakly to a limit $\widehat q$, in
$\sigma(L^p(\Omega),{L^p(\Omega)}')$ when $p\in(2,\infty)$, or weakly$-*$ in $\sigma(L^\infty (\Omega),L^1(\Omega))$ when $p=\infty$. 
The limit element $\widehat q$ is included in $\mathbb{Q}_{ad}$ as we saw in the ROM based case.
We again conclude the strong convergence of the wavefields using Theorem \ref{th: q mapsto u(q)}: 
\begin{align}
    u_i^{(s)}(q_\nu) \to u^{(s)}_i(\widehat q),  \text{  in } H^1(\Omega), \text{ as } \ \nu \to \infty, \ \nu \in N_1,
\end{align}
for $i=1,...,N, \ s=1,...,M$.
%Furthermore, $Q$ is closed and convex, therefore weakly closed, so we also obtain that $m^\star \in Q$.
Finally, for all $i=1,...,N, \ r,s=1,...,M,$ we obtain 
\begin{align*}
     \lim_{\nu \in N_1} |\langle f^{(r)},u^{(r)}_i(q_\nu) -\mathcal E ^{(r,s)}_i\rangle|=|\langle f^{(r)},u^{(r)}_i(\hat q) -\mathcal E ^{(r,s)}_i\rangle|.
\end{align*}
Therefore, by the strong convergence of the data-fit terms and the weak lower semicontinuity of the norm, we obtain
\begin{align}
    \zeta=\lim_{\nu\in N_1} \phi_{\text{ROM}}(q_\nu;a) 
    \geq 
    \liminf_{\nu\in N_1} \phi_{\text{ROM}}(q_\nu;a)  \geq 
    \liminf_{\nu \in N_1} \phi_{\text{ROM}}(\widehat q;a) = \zeta,
\end{align}
meaning that $\widehat q$ is a local minimizer.
\end{proof}

\section{Conclusions \label{sec: conclusions}}
In this paper we formulated and studied a new, Lippmann-Schwinger type operator equation based on the variational form of the Helmholtz impedance boundary value problem. Using the properties of the variational Lippmann-Schwinger operator we showed weak-to-strong convergence for the parameter-to-wavefield map and established existence of minimizers for frequency domain reduced order model based and conventional waveform inversion methods under $L^p, \ p\in  (2,\infty]$ type regularization. Also, given our proposed framework and results, it is straightworward to show the existence of minimizers in case we deal with parameters that lie in spaces that are continuously embedded in $L^p, \ p\in (2,\infty]$.

One important next step for future research is to find a way to extend the results of this paper to the case when the parameter belongs to $L^p, p\in [1,2]$. As we saw, we obtain compactness of our variational Lippmann-Schwinger operator using classical regularity theory for the solutions of the Helmholtz equation. However, using similar regularity techniques to extend the tools of this paper might fail when we introduce parameters in $L^p, p\in [1,2]$. Studying the coupled system of integral equations that we outlined in the introduction might be the path to successfully extending our methods to the case when the parameter lies in $L^p, p\in [1,2]$. Regarding the variational Lippmann-Schwinger equation itself, there can be many ways one can use it in the context of solving the inverse boundary value problem.
An interesting one is to combine the proposed variational Lippmann-Schwinger equation with reduced order models in order to obtain a linearized inversion approach aimed to estimate $q.$ In particular, as a first step, reduced order model based internal solutions can be used for the state estimation. Following that, the variational Lippmann-Schwinger  equation can be used in the spirit of classical Born approximation to estimate the medium parameter.
Future research will explore towards this direction.

\appendix
\section{Well posedness of  the Helmholtz problem \label{App: Well posedness}}
\begin{proof}[Proof of proposition \ref{pr: forward problem well posedness}]
For the sake of completeness we include a sketch of the proof. We refer to \cite{Wald2018} and \cite{Bao_2005} for more details.
We define the forms 
$a_1,a_2:H^1(\Omega)^2\to \C$ with
\begin{align}
    a_1(u,v)=\int_\Omega \nabla u \cdot \cl {\nabla v}dx-\imath k\int_{\bd\Omega} u\cl vd\Sigma(x),
\end{align}
\begin{align}
    a_2(u,v)=-\int_\Omega mu\overline{v}dx,   \ u,v\in H^1(\Omega).
\end{align}
We note that $a_1$ is a coercive and bounded form and claim that $a_2$ is bounded (see proposition \ref{prop: bounded a_2}).
For $a_1,$ we denote with $L,U$ the low and upper bounds in the coercivity estimate respectively.
We define the linear Riesz isomorphism,
\begin{align}
    \Phi: H^1(\Omega)\to \overline{H^1(\Omega)}', 
\end{align}
with $\Phi u= (u,\cdot )_{H^1}, u\in H^1(\Omega).$ 
Since $a_1(u,\cdot)$ is an antilinear functional on $H^1(\Omega),$ and using the Riesz representation theorem we define $\mathcal T:H^1(\Omega)\to H^1(\Omega)$ 
with
\begin{align}
    a_1(u,v)=(\mathcal Tu,v)_{H^1}.
\end{align}
$\mathcal T$ is one-to-one onto and we have the estimates
$\|\mathcal T\|_\infty\leq U, \ \| \mathcal T^{-1}\|_{\infty} \leq L.$
Also, we define the linear operator $\mathcal W: L^2(\Omega) \to \overline{H^1(\Omega)}', \ u\stackrel{\mathcal W }{\mapsto} a_2(u,\cdot).$
We also define the linear map
\begin{align}
    \mathcal A_1=\mathcal T^{-1} \Phi^{-1} \mathcal W: L^2(\Omega) \to H^1(\Omega)
\end{align}
and
\begin{align}
    \mathcal A= \mathcal A_1 \circ i_{H^1 \to L^2}:H^1(\Omega) \stackrel{c}{\subspace} L^2(\Omega)\to H^1(\Omega) , \ s\mapsto \mathcal A_1 s.
\end{align}
$\mathcal A$ is bounded as composition of bounded operators.
Also, for $s\in H^1(\Omega), w \in H^1(\Omega)$, we have $a_1(\mathcal As,w)=a_2(s,w).$
We claim that $\mathcal I+k^2 \mathcal A$ is one-to-one.
Let now $y \in H^1(\Omega).$
Finding a solution of the differential equation, is equivalent to finding $y\in H^1(\Omega)$ that satisfies
\begin{align*}
    a_1(y,v)+k^2 a_2(y,v)=\langle  f,v\rangle, \fa v \in H^1
    \stackrel{}{\iff}
    \\
    a_1(y,v)+k^2a_1(\mathcal Ay,v)=\langle f,v\rangle, \ \fa v \in H^1 \iff
\end{align*}
\begin{align}
    a_1(y+k^2\mathcal Ay,v)=\langle  f,v\rangle, \ \fa v \in H^1 \iff 
\end{align}
\begin{align}
    (\mathcal T(y+k^2\mathcal Ay),v)_{H^1(\Omega)}=\langle f,v\rangle, \ \fa v \in H^1 \Ra
\end{align}
\begin{align}
    \Phi \mathcal T (\mathcal I+k^2\mathcal A)y\stackrel{\overline{H^1(\Omega)}'}{=}f \iff  (\mathcal I+k^2 \mathcal A)y=\mathcal T^{-1} \Phi^{-1} f \in H^1(\Omega).
\end{align}
Since $\mathcal A \in \mathcal{L}(H^1(\Omega),H^1(\Omega))$ is compact and $\mathcal I+k^2\mathcal A$ is injective, using the Fredholm alternative we obtain that there exists a unique element $y\in H^1(\Omega) $ that satisfies the last equation. Finally, we obtain the forward stability estimate
\begin{align*}
    \| y \|_{H^1} \leq \| (\mathcal I+k^2\mathcal A)^{-1}\|_{\mathcal{L}(H^1,H^1)}  \| \mathcal T^{-1}\|_{\mathcal{L}(H^1,H^1)} \| \Phi^{-1}\|_{\mathcal{L} (H^1,\overline{H^1}')} \| f \|_{\overline{H^1}'}.
\end{align*}
\end{proof}
\begin{proposition}\label{prop: bounded a_2}
    Given $q$ as in assumption \ref{assumption}, $a_2$ is a bounded form on $H^1(\Omega).$
\end{proposition}
\begin{proof}
We remind the reader that $m=1+q.$ We get
\begin{align}
\nonumber
    |a_2(u,v)|=|(mu,v)_{L^2}|=|(u,mv)_{L^2}| \leq \\ \| u\|_{L^2} \|m v\|_{L^2} \leq \| u\|_{L^2} \| m\|_{L^p
    } \| v\|_{L^s}  \leq C \| m\|_{L^p
    } \|u\|_{L^2} \| v\|_{H^1},
    \nonumber
\end{align}
with $s,p$ connected via the generalized Hölder inequality as
$sp/(p+s)$, or $s=2p/(p-2)$ if $p\in (2,\infty)$ or $s=2,p=\infty.$
We also obtain that
\begin{align}
    \|a_2(u,\cdot)\|_{\overline{H^1}'} \leq \| m\|_p \|u\|_{L^2}\Ra \| \mathcal W \|_{\mathcal{L}(L^2,\overline{H^1}')}\leq \| m \|_{L^p}.
\end{align}
\end{proof}
\section{Collectively Compact Operators \label{App: Collectively Compact Op}}
Here we name a few results regarding collectively compact operators. For more details refer to \cite{kresslinear}.
\begin{definition}
    A collection $\mathcal A =\{ A_i: X\to Y, i\in I \}$ of linear operators mapping a normed space $X$ to a normed space $Y$ is called collectively compact if for each bounded set $U\subset X$ the set 
    \begin{align}
        \{A_i \phi: \phi \in U, i\in I \}
    \end{align}
    is relatively compact in $Y.$
\end{definition}
\begin{remark}
    Given a sequence of collectively compact operators $\{A_n\}_{n\in \N}$ with 
    \begin{align}
        A_nx\to Ax, \ n\to \infty,
    \end{align}
    then the limit operator is compact.
\end{remark}
In the following, the operator norm will be denoted as $\|\cdot\|_\infty$.
\begin{theorem}
    Let $X,Y,Z$ be Banach spaces and $\{A_i\}_{i\in I}:X\to Y$ be a set of collectively compact operators. Let  $L_n:Y\to Z,  \ n \in \N$ be a a sequence of operators that converges pointwise to $L$ as
    \begin{align}
        L_ny\to Ly, \  n\to \infty, \  y\in Y.
    \end{align}
    Then 
    \begin{align}
        \|(L_n-L) A_i\|_{\infty} \to 0, \ \ n\to \infty,
    \end{align}
    uniformly, for $i \in I$, or equivallently
    \begin{align}
        \sup_{i \in I}  \|(L_n-L) A_i\|_{\infty}\to 0, \ n \to \infty. 
    \end{align}
\end{theorem}
Since one compact operator can be viewed as a set containing only one element, we obtain the following. 
\begin{corollary}
Let $X,Y,Z$ be Banach spaces,  and let $L_n:Y\to Z,  \ n \in \N$  be a sequence of operators  that converges pointwise to $L$ for $y\in Y$. Then given a compact operator $A:X\to Y$ we obtain
\begin{align}
        \|(L_n-L) A\|_{\infty} \to 0, \ \ n\to \infty.
    \end{align}
\end{corollary}
\begin{theorem}
\label{th: collectively compact theorem}
    Let \( A : X \to X \) be a compact linear operator on a Banach space \( X \) and let \( I - A \) be injective (thus by Riesz-Fredholm invertible). Assume the sequence of collectively compact operators $ A_n : X \to X, n\in \N $ that converges pointwise \( A_n y \to A y \), for all \( y \in X \). Then for sufficiently large \( n \), more precisely for all \( n \) with
\[
\| (I - A)^{-1} (A_n - A) A_n \|_\infty < 1,
\]
the inverse operators \( (I - A_n)^{-1} : X \to X \) exist and are bounded by
\begin{align}
   \| (I - A_n)^{-1} \|_\infty \leq \frac{1+\|(I+A)^{-1}A_n\|_\infty}{1 - \| (I - A)^{-1} (A_n - A) A_n \|_\infty}. 
\end{align}

For the solutions of the equations
\[
y-A y = f \quad \text{and} \quad y_n - A_n y_n = f_n,
\]
we have the error estimate
\begin{align}
    \| y_n - y \|_X \leq \frac{1+\|(I+A)^{-1}A_n\|_\infty}{1 - \| (I - A)^{-1} (A_n - A) A_n \|_\infty}
\left( \|  (A_n - A) y \|_X + \| f_n - f \|_X \right).
\end{align}
\end{theorem}
\begin{corollary}
    \label{cor: colectively compact estimate}
    For sufficiently large $n$, we obtain 
\begin{align}
    \| y_n - y \|_X \leq C
\left( \|  (A_n - A) y \|_X + \| f_n - f \|_X \right),
\label{rel: colectively compact estimate}
\end{align}
with $C$ not depending on $n$.
\end{corollary}
\section{Extra results \label{App: Extra}}
\begin{proof}[Proof of proposition \ref{prop: q_n to q means q geq 0}]
%We claim that the set $ \{q\in L^p(\Omega): q\geq 0\}$ is closed. It now follows that the set
%$\mathbb Q= \{q\in L^p(\Omega): q\geq -1\}$ is also closed.
%We now prove the claim.
First, we assume that $2<p<\infty.$ Given a compactly supported smoth function $y\in C^\infty_c(\Omega;[0,\infty)),$ then 
    \begin{align}
       \fa n \in \N : \int_\Omega y (1+q_n) dx \geq 0,
    \end{align}
    hence
    \begin{align}
       \lim _{n\to \infty } \langle y, 1+q_n\rangle_{(L^{p})'}= \lim_n  \int_\Omega y (1+q_n) dx \geq 0 \Ra  \lim _n \langle y, 1+q_n\rangle_{(L^p)'} =\langle y, 1+q\rangle_{(L^{p})'} \geq 0,
    \end{align}
    thus 
    \begin{align}
        \langle y, 1+q\rangle_{(L^{p})'}=   \int_\Omega y(1+ q) dx \geq 0
    \end{align}
    for all $y\geq 0$ test functions. 
    %% ADD Ok rephrase
    Thus, if $1+q$ were negative on a subset of our domain of positive measure, we could construct a function supported on the same set, and the integral would still turn out to be positive.
    
    %thus if $q$ was negative for a positive measure subset of our domain, then we could define a function with the same set as support but that would still yield that the integral is positive. 
    
    If $p=\infty$ then it is straightforward to see that the set 
    $ \{q\in L^\infty(\Omega): q\geq -1\}$
    is closed for the topology $\sigma(L^\infty(\Omega),L^1(\Omega))$, see a similar case in \cite[page 85]{lions1971optimal}. The same conclusion can also be reached by applying the arguments used in the case $p\not = \infty$.
\end{proof}

\bibliographystyle{siamplain}
\bibliography{bibliography}
\end{document}